\documentclass[12pt]{amsart}
\usepackage{amsfonts, amsbsy, amsmath, amssymb}

\newtheorem{thm}{Theorem}[section]

\newtheorem{thm-con}[thm]{Theorem-Conjecture}
\numberwithin{equation}{section}

\theoremstyle{definition}

\newcommand{\f}{\Bbb F}

\begin{document}

\title[A Note on Permutation Binomials and Trinomials over Finite Fields]{A Note on Permutation Binomials and Trinomials over Finite Fields}

\author[Neranga Fernando]{Neranga Fernando}
\address{Department of Mathematics,
Northeastern University, Boston, MA 02115}
\email{w.fernando@northeastern.edu}

\begin{abstract}
Let $p$ be an odd prime and $e$ be a positive integer. We completely explain the permutation binomials and trinomials arising from the reversed Dickson polynomials of the $(k+1)$-th kind $D_{n,k}(1,x)$ over $\f_{p^e}$ when $n=p^l+2$, where $l\in \mathbb{N}$. 
\end{abstract}



\maketitle


\section{Introduction}

Let $p$ be a prime and $e$ be a positive integer. Let $\f_{p^e}$ be the finite field with $p^e$ elements. A polynomial $f \in \Bbb F_{p^e}[{\tt x}]$ is called a \textit{permutation polynomial} (PP) of $\Bbb F_{p^e}$ if the associated mapping $x\mapsto f(x)$ from $\f_{p^e}$ to $\f_{p^e}$ is a permutation of $\Bbb F_{p^e}$.

Let $k$ be an integer such that $0\leq k\leq p-1$ and $a\in \f_{p^e}$. The $n$-th reversed Dickson polynomial of the $(k+1)$-th kind $D_{n,k}(a,x)$ is defined by
\begin{equation}\label{E1.1}
D_{n,k}(a,x) = \sum_{i=0}^{\lfloor\frac n2\rfloor}\frac{n-ki}{n-i}\dbinom{n-i}{i}(-x)^{i}a^{n-2i},
\end{equation}

and $D_{0,k}(a,x)=2-k$ ; see \cite{Wang-Yucas-FFA-2012}.

The premutation behaviour of the $n$-th reversed Dickson polynomial of the $(k+1)$-kind $D_{n,k}(a,x)$ over finite fields and its properties were explored by the author of the present paper in \cite{Fernando-2016-3}. It was shown in \cite{Fernando-2016-3} that to discuss the permutation property of $D_{n,k}(a,x)$, one only has to consider $a=1$. The cases $n=p^l$, $n=p^l+1$, and $n=p^l+2$, where $p$ is an odd prime and $l\geq 0$ is an integer, were discussed in \cite{Fernando-2016-3}. The first two cases were completely explained and we list the results of the last case obtained in \cite{Fernando-2016-3} below. 

\textbf{Result 1.} (see \cite[Remark~2.14]{Fernando-2016-3})
Let $p$ be an odd prime, $k=0$, and $l=e$. Then we have
\[
\begin{split}
D_{p^e+2,0}(1,x)&= \frac{1}{2}\,(1-4x)^{\frac{p^e+1}{2}}-x+\frac{1}{2}
\end{split}
\]
which is a PP of $\f_{p^e}$ if and only if $p^e\equiv 1 \pmod{3}$.

\textbf{Result 2.}(see \cite[Theorem~2.15]{Fernando-2016-3})
Let $p$ be an odd prime and $k=2$. Then $D_{p^l+2,k}(1,x)$ is a PP of $\f_{p^e}$ if and only if $l=0$. 

\textbf{Result 3.}(see \cite[Theorem~2.16]{Fernando-2016-3})
Let $p>3$ and $k=4$. Then $D_{p^l+2,k}(1,x)$ is a PP of $\f_{p^e}$ if and only if the binomial $x^{\frac{p^l-1}{2}}-\frac{1}{2}x$ is a PP of $\f_{p^e}$. 

\textbf{Result 4.}(see \cite[Theorem~2.17]{Fernando-2016-3})
Let $p$ be an odd prime, $n=p^l+2$, and $k\neq 0, 2, 4$. Then $D_{n,k}(1,x)$ is a PP of $\f_{p^e}$ if and only if the trinomial $(4-k)\,x^{\frac{p^l+1}{2}}+k\,x^{\frac{p^l-1}{2}}+(2-k)x$ is a PP of $\f_{p^e}$. 

This paper is a result of a question asked by Xiang-dong Hou. He asked the author (private communication) ``when is the trinomial in result 4 a PP of $\f_{p^e}$?''. This paper answers that question completely. 

The paper is organized as follows. 

In Section 2, we present some preliminaries that will be used throughout the paper. In Section 3, we explain a family of permutation trinomials over $\f_{p^e}$ arising from the reversed Dickson polynomials when $p>3$ is odd and $k$ is an integer such that $k\neq 0,2,4$. In Section 4, we explain a family of permutation binomials over $\f_{p^e}$ arising from the reversed Dickson polynomials when $p=3$.


\section{Preliminaries}

In this section, we list some preliminaries that will be useful in latter sections. Let $q=p^e$ in the following two theorems. 

\begin{thm}(Hermite's Criterion, see \cite{Lidl-Niederreiter-97}). $f\in \Bbb F_{q}[x]$ is a permutation polynomial of $\Bbb F_{q}$ if and only if the following two conditions hold:
\begin{enumerate}
\item [(i)]$f^{q-1}\pmod{x^q-x}$ has degree $q-1$;
\item [(ii)]for each integer $s$ with $1\leq s\leq q-2$, $f^s\equiv f_s\pmod{x^q-x}$ for some $f_s\in \Bbb F_{q}[x]$ with deg $f_s\leq q-2$.
\end{enumerate}
\end{thm}

\begin{thm} \label{L1} (see \cite{Zieve-2009})
Pick $d, r> 0$ with $d \mid (q-1)$, and let $h\in \f_q[x]$. Then $f(x)=x^r h(x^{\frac{q-1}{d}})$ permutes $\f_q$ if and only if 
\begin{enumerate}
\item $\textnormal{gcd}(r, \frac{q-1}{d})=1$, and
\item $x^r h(x)^{\frac{q-1}{d}}$ permutes $\mu_d$.
\end{enumerate}
\end{thm}


\section{A family of permutation trinomials}

In this section, we completely explain the permutation behaviour of the trinomial in result 4 when $p>3$. 

\begin{thm}
Let $p>3$ be an odd prime and $q=p^e$, where $e$ is a non-negative integer. Let $k$ be an integer such that $k\neq 0,2,4$ and $0\leq k\leq p-1$. Let 
$$f(x)=(4-k)x^{\frac{p^l+1}{2}}+kx^{\frac{p^l-1}{2}}+(2-k)x.$$

Then $f(x)$ is a PP of $\f_q$ if and only if $l=0$ and $k\neq 3$. 

\end{thm}

\begin{proof}
Assume $l=0$ and $k\neq 3$. 

Since $l=0$, $f(x)=2(3-k)x+k$. Since $k\neq 3$, clearly $f(x)$ is a PP of $\f_q$. 

Consider the following three cases. 

\textbf{Case 1.} $l=0$ and $k=3$

\textbf{Case 2.} $l\neq 0$ and $k=3$

\textbf{Case 3.} $l\neq 0$ and $k\neq 3$

Now we claim that $f(x)$ is not a PP of $\f_{p^e}$ in each case above. 

\textbf{Case 1.} Since $l=0$ and $k=3$, $f(x)=3$, which is clearly not a PP of $\f_{p^e}$.

\textbf{Case 2.} Let $l\neq 0$ and $k=3$. Then 

$$f(x)=x^{\frac{p^l+1}{2}}+3x^{\frac{p^l-1}{2}}-x.$$

Note that $f(1)=3$ and 

$$f(-1)=(-1)^{\frac{p^l+1}{2}}+3(-1)^{\frac{p^l-1}{2}}+1.$$

which implies

$$
f(-1) = \,\left\{
        \begin{array}{ll}
            -1 & , \frac{p^l+1}{2}\,\,\textnormal{is even},\\[0.3cm]
             3 & , \frac{p^l+1}{2}\,\,\textnormal{is odd}.
        \end{array}
    \right.
$$

Clearly, $f(x)$ is not a PP when $\frac{p^l+1}{2}$ is odd. 

When $\frac{p^l+1}{2}$ is even, we have $p^l+1\equiv 0 \pmod{4}$ and hence $p>5$. It is clear that $f(4)\equiv 3 \pmod{p}$ when $\frac{p^l+1}{2}$ is even. 

Hence $f(x)$ is not a PP of $\f_{p^e}$ in Case 2. 

\textbf{Case 3.} Let $l\neq 0$ and $k\neq 3$. Consider 

$$f(x)=(4-k)x^{\frac{p^l+1}{2}}+kx^{\frac{p^l-1}{2}}+(2-k)x.$$

Note that $f(0)=0$. Also note that $f(\f_p)\subseteq \f_p$. 

\textbf{Sub Case 3.1.} $l=(2n)e$, where $n\in \mathbb{Z}^+$. Then we have

$$\displaystyle\frac{p^{l}+1}{2}=\displaystyle\frac{p^{(2n)e}+1}{2}=\displaystyle\frac{(p^{ne}-1)(p^{ne}+1)}{2}+1\equiv 1 \pmod{p^e-1},$$

which implies 

$$f(x)=(4-k)x^{\frac{p^l+1}{2}}+kx^{\frac{p^l-1}{2}}+(2-k)x\equiv 2(3-k)x+kx^{p^e-1} \pmod{x^{p^e}-x}.$$

Since 

$$2(k-3)x\equiv k \pmod{p}$$

has a non-zero solution, there exists a non-zero $x\in \f_p$ such that $f(x)=0$. Hence $f(x)$ is not a PP of $\f_{p^e}$.

(or by Hermite's criterion, $f(x)$ is not a PP of $\f_{p^e}$ since the degree of $f(x) > p^e-2$).

\textbf{Sub Case 3.2.}  $l\neq (2n)e$, where $n\in \mathbb{Z}^+$. Note that here we only need to consider $1\leq l\leq 2e-1$ since 

$$\displaystyle\frac{p^{2e+i}+1}{2} \equiv \displaystyle\frac{p^{i}+1}{2} \pmod{p^e-1}\,\,\,\,\textnormal{and}\,\,\,\,\displaystyle\frac{p^{2e+i}-1}{2} \equiv \displaystyle\frac{p^{i}-1}{2} \pmod{p^e-1},$$

which imply

$$(4-k)x^{\frac{p^{2e+i}+1}{2}}+kx^{\frac{p^{2e+i}-1}{2}}+(2-k)x\equiv (4-k)x^{\frac{p^{i}+1}{2}}+kx^{\frac{p^{i}-1}{2}}+(2-k)x\pmod{x^{p^e}-x}.$$

So, let $1\leq l\leq 2e-1$ and consider 

$$f(x)=(4-k)x^{\frac{p^l+1}{2}}+kx^{\frac{p^l-1}{2}}+(2-k)x.$$

Let $x\in \f_p^*$ and $l$ be even. Then 

\[
\begin{split}
f(x)&=(4-k)x^{\frac{p^l-1}{2}+1}+kx^{\frac{p^l-1}{2}}+(2-k)x \cr
&=2(3-k)x+k.
\end{split}
\]

Since 

$$2(k-3)x\equiv k \pmod{p}$$

has a non-zero solution, there exists a non-zero $x\in \f_p$ such that $f(x)=0$. Hence $f(x)$ is not a PP of $\f_{p^e}^*$.

Assume that $l$ is odd. Then clearly for $x\in \f_{p^e}\setminus \f_p$, $f(x)\not \in \f_p$. 

Since $l$ is odd, for $x\in \f_p$ we have 

$$f(x)=(4-k)x^{\frac{p+1}{2}}+kx^{\frac{p-1}{2}}+(2-k)x$$

which implies

$$f(x)^2 = k^2x^{p-1} +\,\, \textnormal{terms with lower degree}.$$

By Hermite's criterion, $f(x)$ does not permute $\f_p$. Hence $f(x)$ is not a PP of $\f_{p^e}$.

This completes the proof. 

\end{proof}


\section{A family of permutation binomials}

In this section, we completely explain the permutation behaviour of the trinomial in result 4 when $p=3$. 

Let $p=3$. Since $k\neq 0,2$, we have $k=1$. Then 
$$f(x)=(4-k)x^{\frac{p^l+1}{2}}+kx^{\frac{p^l-1}{2}}+(2-k)x=x^{\frac{p^l-1}{2}}+x.$$

\begin{thm}
Let $p=3$ and $q=3^e$, where $e$ is a non-negative integer. Let 
$$f(x)=x^{\frac{p^l-1}{2}}+x.$$

Then $f(x)$ is a PP of $\f_q$ if and only if 
\begin{itemize}
\item [(i)]  $l=0$, or
\item [(ii)] $l=me+1$, where $m$ is a non-negative even integer.
\end{itemize}

\end{thm}

\begin{proof}
When $l=0$, $f(x)=x+1$ which is a PP of $\f_q$.

Let $l=me+1$, where $m$ is a non-negative even integer. Note that since $m$ is a non-negative even integer, we have
$$\displaystyle\frac{3^{me+1}-1}{2}=\displaystyle\frac{3^{me}-1}{2}+3^{me}=\displaystyle\frac{(3^{\frac{me}{2}}-1)(3^{\frac{me}{2}}+1)}{2}+3^{me}\equiv 1 \pmod{3^e-1}.$$

So, when $l=me+1$, where $m$ is a non-negative even integer, $f(x)\equiv 2x \pmod{x^{3^e}-x}$ which is clearly a PP of $\f_{3^e}$.

Now assume that $l\neq 0$ and $l=me+1$, where $m$ is a non-negative odd integer.

$$
\displaystyle\frac{3^{me+1}-1}{2} \pmod{3^e-1} \,\,\,\,\textnormal{is}\,\,\,\, \left\{
        \begin{array}{ll}
            \textnormal{even} & , \textnormal{e is odd},\\[0.3cm]
            \textnormal{odd} & , \textnormal{e is even}.
        \end{array}
    \right.
$$

If $\displaystyle\frac{3^{me+1}-1}{2} \pmod{3^e-1}$ is even, then $f(0)=0=f(-1)$ which implies $f(x)$ is not a PP of $\f_q$. 

Now consider the case where $\displaystyle\frac{3^{me+1}-1}{2} \pmod{3^e-1}$ is odd. Note that in this case $e$ is even. Also, 

$\displaystyle\frac{3^{me+1}-1}{2} \equiv \displaystyle\frac{3^e-1}{2}+1 \pmod{3^e-1}$. Then 
$$f(x)=x^{\frac{3^{me+1}-1}{2}}+x \equiv x^{\frac{3^e-1}{2}+1}+x\pmod{x^{3^e}-x}.$$

Now we claim that $x^{\frac{3^e-1}{2}+1}+x$ is not a PP of $\f_{3^e}$. 

$$x^{\frac{3^e-1}{2}+1}+x=x(x^{\frac{3^e-1}{2}}+1)=x\,h(x^{\frac{3^e-1}{2}}),$$

where $h(x)=x+1$. $\textnormal{gcd}(1, \frac{3^e-1}{2})=1$, but $-1 \in \mu_2$ and $h(-1)=0$. So $x\,h(x)^{\frac{3^e-1}{2}}$ does not permute $\mu_2$. Then by Theorem~\ref{L1}, $f(x)$ is not a PP of $\f_{3^e}$. 

This completes the proof. 

\end{proof}


\end{document}